\documentclass[3p,times]{elsarticle}
\usepackage{amsthm}
\usepackage{amsmath}

\usepackage{a4wide}
\usepackage{amssymb}
\usepackage{mathrsfs}
\usepackage[hidelinks]{hyperref}
\usepackage{graphicx}
\usepackage{color}
\usepackage{amsfonts}

\theoremstyle{remark}
\newtheorem*{remark}{Remark}
\theoremstyle{plain}
\newtheorem{thm}{Theorem}[section]
\newtheorem{lem}[thm]{Lemma} 
\newtheorem{claim}[thm]{Claim} 

\begin{document}
\begin{frontmatter}
\title{Fractal property of the graph homomorphism order}
\author[KAM]{Ji\v{r}\'{\i} Fiala\fnref{g1}}
\ead{fiala@kam.mff.cuni.cz}
\fntext[g1]{Supported by M\v{S}MT \v{C}R grant LH12095 and GA\v{C}R grant P202/12/G061.}
\author[IUUK]{Jan Hubi\v{c}ka\fnref{g2}}
\ead{hubicka@iuuk.mff.cuni.cz}
\fntext[g2]{Supported by grant ERC-CZ LL-1201 of the Czech Ministry of Education and CE-ITI P202/12/G061 of GA\v CR.}
\author[China]{Yangjing Long\fnref{g3}}
\ead{yjlong@sjtu.edu.cn}
\fntext[g3]{Supported by National Natural Science Foundation of China (No. 11271258) and Postdoctoral Science Foundation of China (No. 2016M601576).}
\author[IUUK]{Jaroslav Ne\v{s}et\v{r}il\fnref{g2}}
\ead{nesetril@iuuk.mff.cuni.cz}
\address[KAM]{Department of Applied Mathematics\\ Charles University\\ Prague, Czech Republic}
\address[IUUK]{Computer Science Institute\\ Charles University\\ Prague, Czech Republic}
\address[China]{School of Mathematical Sciences\\ Shanghai Jiao Tong University\\Shanghai, China}
\begin{abstract}
We show that every interval in the homomorphism order of finite undirected graphs is either universal or a gap.
Together with density and universality this ``fractal'' property contributes to the
spectacular properties of the homomorphism order.

We first show the fractal property by using Sparse Incomparability Lemma and then by a more involved elementary argument.
\end{abstract}
\end{frontmatter}

%
%
%

\section{Introduction}
In this note we consider finite simple graphs and countable partial orders. On these graphs we consider
all homomorphisms between them. 
Recall that for graphs $G=(V_G,E_G)$ and $H=(V_H,E_H)$ a \emph{homomorphism $f:G\to H$} is an edge preserving mapping $f:V_G\to V_H$, that is:
$$\{x,y\}\in E_G \implies \{f(x),f(y)\}\in E_H.$$
If there exists a homomorphism from graph $G$ to $H$, we write $G\to H$.

Denote by $\mathscr C$ the class of all finite simple undirected graphs without loops and multiple edges, and by $\leq$ the following order:
$$G\to H\iff G\leq H.$$
$(\mathscr C,\leq)$ is called the {\em homomorphism order}.

The relation $\leq$ is clearly a quasiorder which becomes a partial order when factorized by homomorphism equivalent graphs.
This homomorphism equivalence takes particularly simple form, when we represent each class by the so called core. Here,
a \emph{core} of a graph is its minimal homomorphism equivalent subgraph.
It is well known that up to an isomorphism every equivalence class contains a unique core~\cite{Hell2004}.
 However, for our purposes it is irrelevant whether we consider $(\mathscr C,\leq)$ as a quasiorder or a partial order. For brevity we speak of the homomorphism order in both cases.

The homomorphism order has some special properties, two of which are expressed as follows:

\begin{thm}
\label{thm:univ1}
$(\mathscr C,\leq)$ is (countably) universal.

Explicitly: For every countable partial order $P$ there exists an embedding of $P$ into $(\mathscr C,\leq)$.
\end{thm}
Here an {\em embedding} of partial order $(P,\leq)$ to partial order $(P',\leq')$ is an injective function $f:P\to P'$
such that for every $u,v\in P$, $u\leq v'$ if and only if $f(u)\leq'f(v)$.

\begin{thm}
\label{thm:dense}
$(\mathscr C,\leq)$ is dense.

Explicitly: For every pair of graphs $G_1<G_2$ there exists $H$ such that $G_1<H<G_2$.
This holds with the single exception of $K_1<K_2$, which forms the only gap of the homomorphism order of undirected graphs.
\end{thm}
As usual, $K_n$ denotes the complete graph with $n$ vertices. We follow the standard graph terminology as e.g. \cite{Hell2004}.
As the main result of this paper we complement these structural results by the following statement:

\begin{thm}
\label{thm:main}
$(\mathscr C,\leq)$ has the fractal property.

Explicitly: For every pair $G_1< G_2$, distinct from $K_1$ and $K_2$
(i.e. the pair $(G_1,G_2)$ is not a gap), there exists an order-preserving embedding $\Phi$ of $\mathscr C$ into the interval
$$[G_1,G_2]_\mathscr C=\{H;G_1<H<G_2\}.$$
\end{thm}

Putting otherwise, every nonempty interval in $\mathscr C$ contains a copy of $\mathscr C$ itself.

Theorem~\ref{thm:univ1} was proved first in~\cite{hedrlin1969universal,pultr1980combinatorial} and reproved in~\cite{Hubicka2004,Hubicka2005}. Theorem~\ref{thm:dense} was proved in~\cite{welzl1982color} and particularly simpler proof was given by Perles and Ne\v set\v ril~\cite{nesetril1999}, see also~\cite{Nesetril2000,Hell2004}.

Theorem~\ref{thm:main} was formulated in \cite{nesetril1999} and remained unpublished since. The principal ingredient of the proof is the Sparse Incomparability Lemma \cite{nevetvril1989}.
 In addition, we give yet another proof of Theorem~\ref{thm:main}. In fact, we prove all three Teorems~\ref{thm:univ1}, \ref{thm:dense} and \ref{thm:main} (Theorem~\ref{thm:dense} is a corollary of Theorem~\ref{thm:main}).

First, to make the paper self-contained we also give in Section~\ref{sec:univ} a short and easy proof of universality
of $(\mathscr C,\leq)$ which was developed in \cite{Fialab,Fiala2014} and sketched in \cite{Fiala2015}.
Then, in Section~\ref{sec:fractalprop} we give first proof of Theorem~\ref{thm:main} based on the Sparse Incomparability Lemma~\cite{nevetvril1989,Hell2004}.
Then in Section~\ref{sec:secondproof} we prove a strenghtening of Theorem~\ref{thm:dense} (stated as Lemma~\ref{lem:fatgap}). This will be needed for our second proof of Theorem~\ref{thm:main} 
which is flexible enough for applications. Thus this paper summarizes perhaps surprisingly easy proofs of theorems which originally had difficult proofs.

\section{Construction of a universal order}
\label{sec:univ}

\subsection{Particular universal partial order}
\label{ssec:universal}

Let $(\mathcal P,\leq_{\mathcal P})$ be a partial order, where $\mathcal P$ consists of all finite sets of
odd integers, and where for $A,B\in \mathcal P$ we put $A\leq_{\mathcal P} B$ if and only if for every $a\in
A$ there is $b\in B$ such that $b$ divides $a$.  We make use of the following:

\begin{thm}[\cite{Fiala2014}]
\label{thm:universal}
The order $(\mathcal P,\leq_{\mathcal P})$ is a universal partial order.
\end{thm}

To make the paper self-contained we give a brief proof of this assertion.
(See also \cite{hedrlin1969universal,Hubicka2004,Hubicka2011} for related constructions of universal
partial orders.) The proof of Theorem~\ref{thm:universal} follows from two simple lemmas.

We say that a countable partial order is {\em past-finite} if every down-set $x^\downarrow = \{y; y\leq x\}$ is
finite. A countable partial order is {\em past-finite-universal}, if it contains every
past-finite partial order as a suborder.
{\em Future-finite} and {\em future-finite-universal} orders are defined analogously with respect to up-sets $x^\uparrow = \{y; y\geq x\}$.

Let $P_f(X)$ denote the set of all finite subsets of $X$. The following lemma
extends a well known fact about representing finite partial orders by sets
ordered by the subset relation.

\begin{lem}
\label{lem:pastfiniteuniv}
For any countably infinite set $X$, the partial order $(P_f(X),\subseteq)$ is past-finite-universal.
\end{lem}

\begin{proof}
Consider an arbitrary past-finite order $(Q,\leq_Q)$. Without loss of generality we may assume that $Q\subseteq X$. 
Let $\Phi$ be the mapping that assigns every $x\in Q$ its down-set, i.e. $\Phi(x) = \{y\in Q; y\leq x\}$. 
It is easy to verify that $\Phi$ is indeed an embedding $(Q,\leq_Q)\to(P_f(X),\subseteq)$.
\end{proof}

By the {\em divisibility partial order}, denoted by $(\mathbb{N},\leq_d)$, we mean the partial order on positive integers, where $n \leq_d m$ if and only if $n$ is divisible by $m$.
Denote by $\mathbb{Z}_o$ the set of all odd integers $n$, $n\geq 3$.

\begin{lem}
\label{lem:futurefiniteuniv}
The divisibility partial order $(\mathbb{Z}_o,\leq_d)$ is future-finite-universal.
\end{lem}

\begin{proof}
Denote by $\mathbb P$ the set of all odd prime numbers.  Apply Lemma \ref{lem:pastfiniteuniv} for $X=\mathbb P$.
Observe that $A\in P_f(\mathbb P)$ is a subset of $B\in P_f(\mathbb P)$ if and only if $\prod_{p\in A} p$ divides $\prod_{p\in B} p$.
\end{proof}

\begin{proof}[Proof of Theorem \ref{thm:universal}]
Let $(Q,\leq_Q)$ be a given partial order. Without loss of generality we may assume that
$Q$ is a subset of $\mathbb{P}$. This way we obtain also the usual linear order $\leq$ 
(i.e. the comparison by the size) on the elements of $Q$.  
In the following construction, the order $\leq$ determines in which sequence the elements of $Q$ are processed 
(it could be also interpreted as the time of creation of the elements of $Q$).

We define two new orders on $Q$: the \emph{forwarding order} $\leq_f$ and the \emph{backwarding order} $\leq_b$ as follows:
\begin{enumerate}
\item We put $x\leq_f y$ if and only if $x\leq_Q y$ and $x \leq y$.
\item We put $x\leq_b y$ if and only if $x\leq_Q y$ and $x \geq y$.
\end{enumerate}

Thus the partial order $(Q,\leq_Q)$ has been decomposed into $(Q,\leq_f)$ and $(Q,\leq_b)$.
For every vertex $x\in Q$ both sets $\{y; y\leq_f x\}$ and $\{y; x\leq_b y\}$ are finite.
It follows that $(Q,\leq_f)$ is past-finite and that $(Q,\leq_b)$ is future-finite.

Since $(\mathbb{Z}_o,\leq_d)$ is future-finite-universal (Lemma~\ref{lem:futurefiniteuniv}), there is an embedding $\Phi: (Q,\mathop{\leq_b}) \to (\mathbb{Z}_o,\leq_d)$. The desired embedding $U:(Q,\leq_Q) \to (\mathcal P,\leq_{\mathcal P})$ is obtained by representing each $x\in Q$ by a set system $U(x)$ defined by (see Figure~\ref{fig:poset}): 
$$U(x)=\{\Phi(y); y\leq_f x\}.$$

\begin{figure}
\centerline{\includegraphics{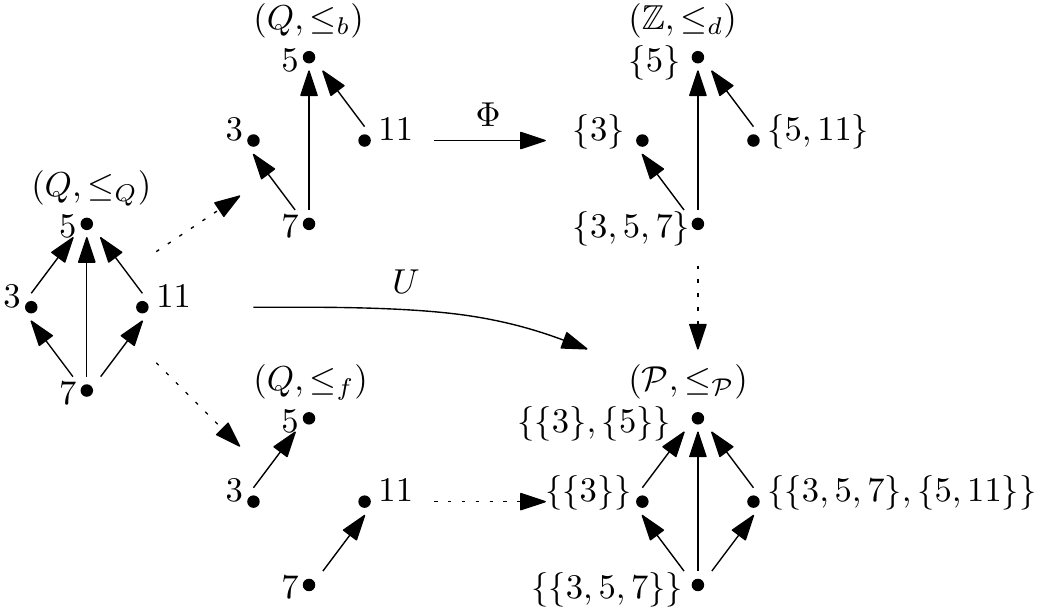}}
\caption{Example of the construction of an embedding $U:(Q,\leq_Q) \to (\mathcal P,\leq_{\mathcal P})$.}
\label{fig:poset}
\end{figure}

To argue the correctness we first show that $U(x)\leq_{\mathcal P} U(y)$ implies $x\leq_Q y$.
From the definition of $\leq_{\mathcal P}$ and the fact that $\Phi(x)\in U(x)$ follows that at least one $w\in Q$ exists, such that
$\Phi(w)\in U(y)$ and $\Phi(x)\leq_d \Phi(w)$. By the definition of $U$, $\Phi(w)\in U(y)$ if and only if $w\leq_f y$. By the definition of $\Phi$, $\Phi(x)\leq_d \Phi(w)$ if and only if $x\leq_b w$.
It follows that $x\leq_b w\leq_f y$ and thus also $x\leq_Q w\leq_Q y$ and consequently $x\leq_Q y$.

To show that $x\leq_Q y$ implies $U(x)\leq_{\mathcal P} U(y)$ we consider two cases.
\begin{enumerate}
\item When $x\leq y$ then $U(x)\subseteq U(y)$ and thus also $U(x)\leq_{\mathcal P} U(y)$.
\item Assume $x>y$ and take any $w\in Q$ such that $\Phi(w)\in U(x)$.  From the definition of $U(x)$ we have
$w\leq_f x$. Since $x\leq_Q y$ we have $w\leq_Q y$.
If $w\leq y$, then $w\leq_f y$ and we have $\Phi(w)\in U(y)$. In the other case if $w>y$ then $w\leq_b y$ and thus $\Phi(w)\leq_d \Phi(y)$. Because the choice of $w$ is arbitrary, it follows that $U(x)\leq_{\mathcal P} U(y)$.
\end{enumerate}%
\end{proof}

Clearly, as in e.g.~\cite{Hubicka2004} this can be interpreted as Alice-Bob game played on finite partial orders. Alice always wins.

\subsection{Representing divisibility}

Denote by $\overrightarrow{C_p}$ the directed cycle of length $p$, i.e. the graph $(\{0,1,\ldots, p-1\},\{(i,i+1); i=0,1,\ldots, p-1\})$, where addition is performed modulo $p$.
Denote by $\mathcal D$ the class of disjoint unions of directed cycles.
\begin{thm}
\label{thm:univ}
The homomorphism order $(\mathcal D,\leq)$ is universal.
\end{thm}

\begin{proof}
Observe first that a homomorphism $f:\overrightarrow{C_p}\to \overrightarrow{C_q}$ between two cycles $\overrightarrow{C_p}$ and $\overrightarrow{C_q}$ exists if and only if $q$ divides $p$. 

Consequently, for two collections of disjoint cycles $\sum_{p\in A}\overrightarrow{C_p}$ and $\sum_{q\in B}\overrightarrow{C_q}$
a homomorphism 
$$f:\sum_{p\in A}\overrightarrow{C_p}\to \sum_{q\in B}\overrightarrow{C_q}$$ 
exists if and only if 
$$A\leq_{\mathcal P} B,$$
with respect to the universal partial order $(\mathcal P,\leq_{\mathcal P})$ of 
Theorem~\ref{thm:universal}. 

Since we have used odd primes in the proof of Lemma~\ref{lem:futurefiniteuniv}, the minimum of each set in $\mathcal P$ is at least three.
Hence, each $A\in {\mathcal P}$ corresponds to a disjoint union of odd cycles.
\end{proof}

\begin{remark}
Denote by $\overrightarrow{\mathscr C}$ the class of all directed graphs.
Theorem~\ref{thm:univ} yields immediately that the order $(\overrightarrow{\mathscr C},\leq)$ is also universal.

The class of disjoint union of directed odd cycles is probably the simplest class for which the homomorphism
order is universal. However note that here the key property is that
objects are not connected and contains odd cycles of unbounded length.
If we want to obtain connected graphs with bounded cycles then we have
to refer to \cite{Hubicka2004,Hubicka2011,Hubicka2005} where it is proved that that the
class of finite oriented trees $\mathcal T$ and even the class of finite oriented
paths form universal partial orders.  These strong notions are not needed in this
paper. However note that from our results here it also follows that not only the class
of planar graphs but also the class of outer-planar graphs form a universal partial
order.
\end{remark}

\section{The fractal property}
\label{sec:fractalprop}

To prove Theorem~\ref{thm:main}, we use the following result proved by Ne\v{s}et\v{r}il and R\"odl by non-constructive methods~\cite{nevetvril1989}. 
Later, non-trivial constructions were given by Matou\v{s}ek-Ne\v set\v ril \cite{matousek2004} and Kun \cite{kun2013}:

\begin{thm}[Sparse Incomparability Lemma \cite{nevetvril1989}, see e.g. Theorem 3.12 of \cite{Hell2004}]
\label{thm:sparse}
Let $l$ be positive integer. For any non-bipartite graphs $G_1$ and $G_2$ with $G_1 < G_2$, there exists a connected graph $F$ such that
\begin{itemize}
\item $F$ is (homomorphism) incomparable with $G_1$ (i.e. $F\not \leq G_1$ and $G_1\not \leq F$);
\item $F < G_2$; and
\item $F$ has girth at least $l$, where the girth of a graph is the length of its shortest cycle.
\end{itemize}
\end{thm}

In the sequel we combine the Sparse Incomparability Lemma and the universality of $(\mathcal D,\leq)$, together with the standard indicator 
technique developed by Hedrl\'{\i}n and Pultr~\cite{hedrlin1964,Hell2004}.

The essential construction of this method takes 
an oriented graph $G$ and a graph $I$ with two distinguished vertices $a$, $b$ and creates a graph $G*(I,a,b)$, obtained by substituting every arc $(x,y)$ of $G$ by a copy of the graph $I$, where $x$ is identified with $a$ and $y$ is identified with $b$, see Figure~\ref{fig:suffices2} for an example. 

\begin{figure}
\centerline{\includegraphics{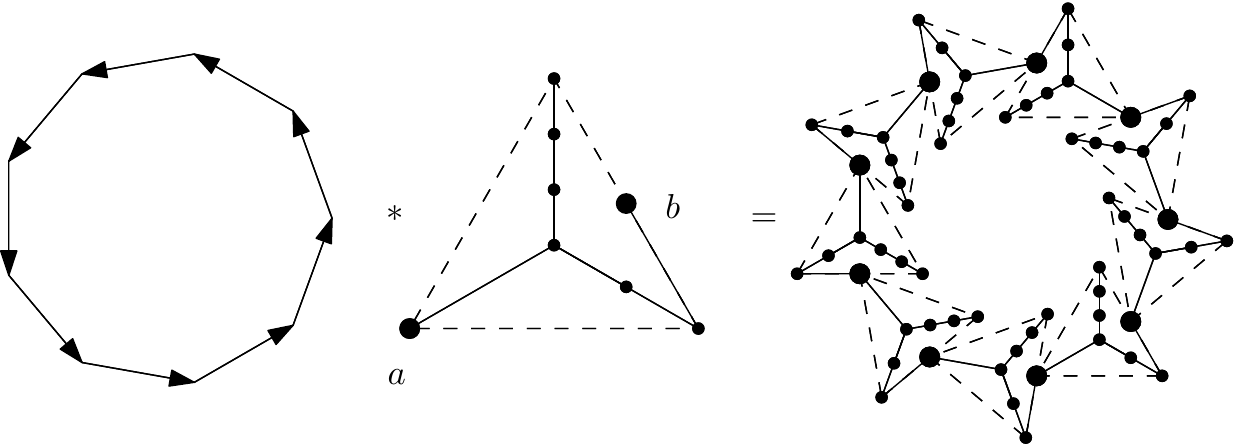}}
\caption{Construction of $\protect\overrightarrow{C_p}*(I,a,b)$ .}
\label{fig:suffices2}
\end{figure}

\begin{proof}[Proof of Theorem \ref{thm:main}]
Let be given undirected graphs $G_1<G_2$ not forming a gap. By our assumptions $G_2$ is a core distinct from $K_2$ as otherwise $G_1=K_1$. (If $G_2=K_2$ then $K_1=K_1$ and we have  a gap). We may also assume without loss of generality that $G_1$ is not bipartite since in such a case we may replace $G_1$ by a graph $G_1<G_1'<G_2$ (given by Theorem~\ref{thm:dense}), making the interval $[G_1',G_2]_\mathscr C$ even narrower than $[G_1,G_2]_\mathscr C$. Because all bipartite graphs are homomorphically equivalent it follows that $G_1'$ is non-bipartite.

Let $l\ge 5$ be any odd integer s.t. the longest odd cycle of $G_2$ has length at most $l$.

For the indicator we use the graph $I_l$, depicted in Figure~\ref{fig:indicator}.
The graph $I_l$ can be viewed either as a subdivision of $K_4$, or as 3 cycles of length $l+2$ amalgamated together.
\begin{figure}
\centerline{\includegraphics{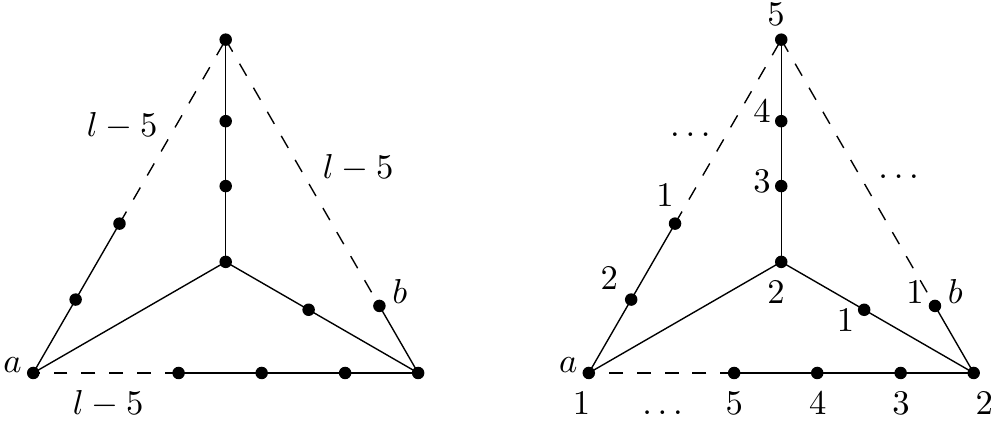}}
\caption{The indicator $I_l$ and its homomorphism to $C_l$. Dashed lines represent paths with $l-5$ internal vertices.} 
\label{fig:indicator}

\end{figure}
The indicator $I_l$ is rigid, i.e. the only homomorphism $I_l\to I_l$ is the identity~\cite[Proposition 4.6]{Hell2004}. Note also that $I_l$ allows a homomorphism to the undirected cycle of length $l$, as is also depicted in Figure~\ref{fig:indicator}. 

We continue with the construction of a graph $H_A$ from a set of odd integers $A\in {\mathcal P}$. 
Let $F$ be a connected graph satisfying the conclusions of the Sparse Incomparability Lemma. We fix an arbitrary vertex $u$ of $F$.

Then, given a positive integer $p\geq 3$, we apply the indicator $I_l,a,b$ on the directed cycle of length $l\cdot p$ to obtain  $\overrightarrow{C_{lp}}*(I_l,a,b)$.
(Observe that $\overrightarrow{C_{lp}} * (I_l,a,b) \to \overrightarrow{C_{lq}} * (I_l,a,b)$ if and only if $q$ divides $p$.)
 We then join any vertex of the original cycle $\overrightarrow{C_{lp}}$ to $u$ by a path of length $|V_{G_2}|$, see Figure~\ref{fig:graphH}.  

\begin{figure}
\centerline{\includegraphics{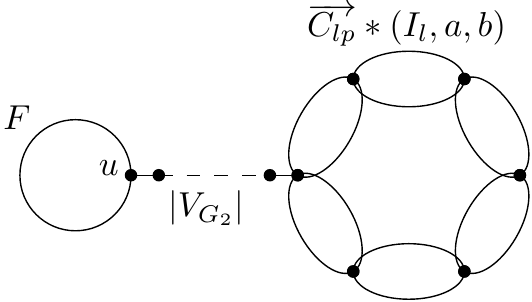}}
\caption{Graph $H_p$.}
\label{fig:graphH}
\end{figure}

Observe that the resulting graph $H_p$ allows a homomorphism to $G_2$,
since:
\begin{enumerate}
\item There exists a homomorphism $f:F \to G_2$ by Theorem~\ref{thm:sparse};
\item the indicator $I_l$ has a homomorphism to a cycle of length $l$, which can be simply transformed to a homomorphism $g$ to any odd cycle of length $l'\le l$ in $G_2$ (by the choice of $l$);
\item the mapping $g$ could be chosen s.t. $g(a)=g(b)$, hence it can be extended to all vertices of $\overrightarrow{C_p}$; 
\item the distance between the image of $u$ and the cycle of length $l'$ is at most $|V_{G_2}|$, therefore both homomorphisms $f$ and $g$ can be combined together and extended to the whole graph $H_p$ straightforwardly.
\end{enumerate}

To complete the construction of $H_A$, we put 
$$H_A=\sum_{p\in A} H_p + G_1.$$ 

The conclusion of Theorem~\ref{thm:main} follows from the following three properties:

\begin{enumerate}
\item For every $A\in {\mathcal P}: G_1 < H_A$.

The $\le$ inequality is obvious, since $G_1$ is a component of each $H_A$.

Since $F$ is a subgraph of each $H_p$, by Theorem~\ref{thm:sparse} there is no homomorphism $H_A\to G_1$ whenever $A$ is nonempty.

\item For every $A\in {\mathcal P}: H_A < G_2$.

The existence of homomorphisms $H_p\to G_2$ and $G_1 \to G_2$ yields a homomorphism $H_A \to G_2$.

As $G_2 \not\leq F$, and as the shortest cycle in $\overrightarrow{C_{lp}}*(I_l,a,b)$ has length $l+2$, which is by the choice of $l$ longer than the length of any odd cycle in $G_2$, there is no homomorphism $G_2 \to H_A$.
 
\item For every $A, B \in {\mathcal P}: H_A \to H_B$ if and only if $A\leq_{\mathcal P} B$.

It is easy to see that $q$ divides $p$ iff $\overrightarrow{C_{lp}}*(I_l,a,b)\to
\overrightarrow{C_{lq}}*(I_l,a,b)$. This is a standard argument.
Note that the paths between $F$ and $\overrightarrow{C_{lp}}$ in $H_p$, and between $F$ and $\overrightarrow{C_{lq}}$ in $H_q$ have the same length and the  
vertex $u$ of attachment has been chosen in the same way in both cases.
Therefore, $H_p\to H_q$ and consequently, $A\leq_{\mathcal P} B$ implies $H_A\to H_B$.

Assume now that $H_A\to H_B$. We have already excluded $H_p \to G_1$, hence by the connectivity of each $H_p$ neccessarily follows that $\sum_{p\in A}H_p\to \sum_{q\in B}H_q$.
This in turns leads to $\sum_{p\in A} \overrightarrow{C_{lp}}*(I_l,a,b)\to \sum_{q\in B}\overrightarrow{C_{lq}}*(I_l,a,b)$ which is equivalent to $A\leq_{\mathcal P} B$.
\end{enumerate}

These three properties guarantee that gave a full embedding of $(\mathcal D,\leq)$ into $(\mathscr C,\leq)$, which maps every $\sum_{p\in I} \overrightarrow{C_p}$ into the interval $[G,G']_\mathscr C$ in $\mathscr C$. By Theorem~\ref{thm:univ} the universality of $[G,G']_\mathscr C$ follows.
\end{proof}

\section{Alternative elementary proof}
\label{sec:secondproof}

The sparse incomparability lemma holds for ``dense'' classes of 
graphs. Here we establish the fractal property of graphs by a different 
technique which allows us to prove the fractal property of some ``sparse'' classes of graphs
as well.
For example we can reduce the (stronger form of) density for planar graphs to 
the fractal property of the class of planar graphs. But first we formulate
the proof for general graphs. We shall make use of the following two assertions.

\begin{lem}
\label{lem:fatgap}
Given graphs $G_1<G_2$, $G_2$ non-bipartite, there exists connected graphs $H_1$ and $H_2$ with
properties:
\begin{enumerate}
 \item $H_1$ and $H_2$ are homomorphically incomparable, and,
 \item $G_1<H_i<G_2$ for $i=1,2$.
\end{enumerate}
In other words, any non-gap interval in the homomorphism order contains
two incomparable graphs.
\end{lem}
\begin{proof}
Proof is a variant of the Ne\v{s}et\v{r}il-Perles' proof of density~\cite{Hell2004}.
Put $$H_1=(G_2\times H)+G_1,$$
where $H$ is a graph that we specify later, $+$ is the disjoint union and $\times$ denotes the direct product.

Then obviously $G_1\leq H_1\leq G_2$. If the odd girth of $G_2$ is larger than the odd
girth of $G_1$ then $G_2\not \to H_1$. If the chromatic number $\chi(H)>| V(G_1)|^{|V(G_2)|}$
then any homomorphism $G_2\times H\to G_1$ induces a homomorphism $G_2\to G_1$ which is absurd (see \cite[Theorem 3.20]{Hell2004}). Thus $G_2\times H\not \to G_1$ and $$G_1<H_1<G_2.$$

We choose $H$ to be a graph with large odd girth and chromatic number (known to exist~\cite{Erdos1960}). This finishes
construction of $H_1$.
(Note that here we use the fact that the odd girth of the product is the maximum of the odd girths of its factors.)
Now we repeat the same argument with the pair $G_1$ and $G_2\times H$, put
$$H_2=(G_2\times H')+G_1.$$
If the odd girth of $H'$ is larger than the odd girth of $G_2\times H$ 
then $G_2\times H\not \to H_2$ (assuming $H$
and thus $G_2\times H$ is connected). Thus in turn $H_1\not \to H_2$.
If $\chi(H')>\max(|V(G_2\times H)|^{|V(G_2)|},|V(G_2)|^{|V(G_1)|})$
then again $G_2\times H'\to G_2\times H$ implies $G_2\to G_2\times H$ which
is absurd.  Similarly $G_2\times H'\to G_1$ implies $G_2\to G_1$ and thus $G_2\times H'\not \to G_2$

We may also assume that graphs $H_1$ and $H_2$ from Lemma~\ref{lem:fatgap}
are connected as otherwise we can join components by a long enough path.
Connectivity also follows from the following folklore fact:

\begin{claim}
\label{claim:paths}
For every connected non-bipartite graph $H$ there exists an integer $l$ such that for any two vertices
$x,y\in V(H)$ and any $l'\geq l$ there exists a homomorphism
$f:P_{l'}\to H$ such that $f(0)=x$ and $f(l')=y$. ($P_{l'}$ is the path
of length $l'$ with vertices $\{0,1,\ldots,l'\}$).
\end{claim}

This concludes the construction of $H_1$ and $H_2$.
\end{proof}

\begin{proof}[Second proof of Theorem~\ref{thm:main}]
Let $G_1<G_2$ be a non-gap, thus $G_2$ is non-bipartite. Assume without loss of generality that $G_2$ is connected. 
Since the homomorphism order is universal, we prove the universality of the interval $(G_1,G_2)$ by embedding the homomorphism
order into it.

Let $H_1$, $H_2$ be two connected graphs given
by Lemma~\ref{lem:fatgap}. We may assume that both $H_1$ and $H_2$ are cores. Let $l$ be the number given by Claim~\ref{claim:paths}
for graph $G_2$. We may assume $l>\max\{|V(H_1)|,\allowbreak |V(H_2)|,\allowbreak | V(G_2)|\}$.
We construct the gadget $I$ consisting of graphs $H_1$, $H_2$ joined together by two paths of length $2l$ and $2l+1$. 
We choose two distinguished vertices $a$, $b$ to be the middle vertices of these two paths, see Figure~\ref{fig:twocycle}. 

\begin{figure}
\centerline{\includegraphics{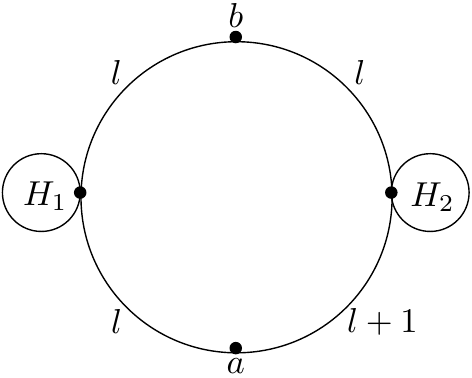}}
\caption{Gadget $I$.}
\label{fig:twocycle}
\end{figure}

We observe that any homomorphism $f:I\to I$ is surjective (because both $G_1$ and $G_2$ are cores), it is also an identity on vertices of $I\setminus (G_1\cup G_2)$ and that 
there exists a homomorphism $f:I\to G_2$ such that $f(a)=f(b)$.

For every oriented graph $G$ define graph
$\Phi(G)$ as $\Phi(G)=G*(I,a,b)$. We know $G_1 < \Phi(G)<G_2$ (as any homomorphisms $f_i:H_i\to G_2$, $i\in \{1,2\}$, can be extended to a homomorphism $\Phi(G)\to G_2$).

We finish the proof by proving $\Phi(G)\to \Phi(G')$ if and only if $G\to G'$.

Assume first that there exists a homomorphism $f:G\to G'$. Consider the function $g$ defined as $f$ on vertices of $G$ and as the unique mapping which maps a copy of $(I,a,b)$ in $G$ corresponding to edge $(u,v)$ to the copy of $(I,a,b)$ in $G'$ corresponding to edge $(f(u),f(v))$. Hence $g$ is a homomorphism.

Let now $g$ be a homomorphism $\Phi(G)\to\Phi(G')$. By the girth assumption and connectivity of $H_1$ and $H_2$ we know that $g$ maps every copy of $H_1$ (or $H_2)$) in $\Phi(G)$ to a copy of $H_1$ (or $H_2$) in $\Phi(G')$. Again, by the girth argument it follows that every copy of the indicator $(I,a,b)$ in $G$ is mapped to a copy of the indicator $(I,a,b)$ in $G'$. But the only copies of $(I,a,b)$ in both $G$ and $G'$ are those corresponding to the edges of $G$ and $G'$.

Since $I$ is a core, it follows that any pair of vertices $(a,b)$ in a copy of $(I,a,b)$ has to be mapped to the vertices $(a,b)$ in any other copy of $(I,a,b)$. As copies of $(I,a,b)$ and hence also the pairs $(a,b)$ correspond to edges of $G'$, it follows that $g$ induces a mapping $f:V(G)\to V(G')$, which is a homomorphism $G\to G'$. 
This argument concludes the second proof of Theorem~\ref{thm:main}.
\end{proof}

\begin{remark}
Note that in this second proof we have an one-to-one correspondence between homomorphisms $G\to G$ and $\Phi(G)\to \Phi(G')$.
\end{remark}


\section{Concluding remarks}

\noindent
{\bf 1.} Gaps on oriented graphs and on more general relational structures are more sophisticated. They were characterized by Ne\v{s}et\v{r}il and Tardif~\cite{Nesetril2000}.
In the same paper, a nice 1-1 correspondence between gaps and dualities has been shown.
Consequently, the full discussion of fractal property of relational structures is more complicated and it will appear elsewhere.

\noindent
{\bf 2.} The whole paper deals with finite graphs but there is no difficulty in
generalizing our results to infinite graphs.

\noindent
{\bf 3.} An interesting question (already considered in~\cite{Nesetril2000}) is: which intervals induce {\em isomorphic orders}.
We provide neither a characterization nor a conjecture in this direction.

\noindent
{\bf 4.} It seems that any natural universal class $\mathcal K$ of structures posesses the \emph{gap-universal dichotomy}: An interval $[A,B]_{\mathcal K}$
in $\mathcal K$ either contains a gap or it contains a copy of $\mathcal K$ itself.
While in general this fails, it remains to be seen, whether this is true for some general class of structures.

\noindent
{\bf 5.} There is a great difference in treating universality, density and fractal property.
Whereas the universality was established in many classes of partial orders (and categories), the density and fractal property was only established for just a few basic classes.
Perhaps this should be investigated in greater depth. Apart from general relational structures (which we hope to treat in another paper) another interesting case is
provided by structures with two equivalences (here the universality is established by \cite{Nevsetvril1971}).

\paragraph{Acknowledgement}
We would like to thank to two anonymous referees for comments and suggestions which improved quality of this paper.

\bibliographystyle{elsarticle-num}

\bibliography{fractal}

\end{document}